\documentclass[oneside,leqno,11pt]{amsart}

\title{Borel Tukey morphisms and combinatorial cardinal invariants of
  the continuum}

\author{Samuel Coskey}
\address{Samuel Coskey, Department of Mathematics, Boise State
  University, 1910 University Dr, Boise, ID 83725-1555 (Formerly
  at York University, Toronto, Canada)}
\email{scoskey@nylogic.org}
\urladdr{boolesrings.org/scoskey}

\author{Tam\'as M\'atrai}
\address{Tam\'as M\'atrai, Alfr\'ed R\'enyi Matematikai
  Kutat\'oint\'ezet, Magyar Tudom\'anyos Akad\'emia, 13-15
  Re\`altanoda utca, H-1053 Budapest, Hungary}
\email{matrait@renyi.hu}

\author{Juris Stepr\=ans}
\address{Juris Stepr\=ans, York University Department of Mathematics
  and Statistics, N520 Ross, York University, 4700 Keele St., Toronto,
  ON M3J 1P3, Canada}
\email{steprans@yorku.ca}

\thanks{The authors were partially supported by NSERC}

\usepackage[marginratio=1:1]{geometry}
\usepackage{setspace}
\usepackage{mathpazo}
\usepackage{mydefs}
\usepackage{diagrams}
\usepackage{tikz}\usetikzlibrary{decorations.pathreplacing}

\newcommand{\vojtas}{Vojt\'a\v{s}}
\newcommand{\fr}[1]{\ensuremath{\mathfrak{#1}}}
\newcommand{\boxhead}[1]{\vspace{6pt}\noindent\fbox{#1}\ }
\newcommand{\A}{\mathsf{A}}
\newcommand{\B}{\mathsf{B}}

\begin{document}

\begin{abstract}
  We discuss the Borel Tukey ordering on cardinal invariants of the
  continuum.  We observe that this ordering makes sense for a larger
  class of cardinals than has previously been considered.  We then
  provide a Borel version of a large portion of van~Douwen's diagram.
  For instance, although the usual proof of the inequality $\fr
  p\leq\fr b$ does not provide a Borel Tukey map, we show that in fact
  there is one.  Afterwards, we revisit a result of Mildenberger
  concerning a generalization of the unsplitting and splitting
  numbers.  Lastly, we show that the inclusion ordering on $\mathcal
  P(\omega)$ embeds into the Borel Tukey ordering on cardinal
  invariants.
\end{abstract}

\maketitle

\section{Introduction}

Cardinal invariants of the continuum are cardinal numbers which are
determined by families of real numbers (or any similar continuum such
as $\mathcal P(\omega)$, the set of subsets of the natural numbers).
For instance, the least size of a Lebesgue non-null set is a cardinal
invariant, one of many derived from properties of measure and
category.  A second example is the least size of a family of sequences
of natural numbers such that any other sequence is eventually
dominated by one from the family.  This example is one of several
which are known as \emph{combinatorial} cardinal invariants.

As with each of these examples, most classical cardinal invariants
take on values between $\aleph_1$ and $\fr c$.  The particular values
can vary from one model of set theory to another; for instance, in a
model of \textsf{CH} they always have value $\aleph_1=\fr c$.  But the
pattern of values is not arbitrary: there exist deep connections
between them which dictate that certain inequalities must hold in any
model of set theory.  We refer the reader to Andreas Blass's excellent
article \cite{blass} in the Handbook of Set Theory for a survey of
this rich area of research.

Meanwhile, in this article we will be interested in a categorical
approach to cardinal invariants and their inequalities which is due to
\vojtas.  See \cite{voj}, or Section~4 of \cite{blass} for a more
detailed account.  This approach rests on on the following definition
scheme for cardinal invariants.  A \emph{\vojtas\ triple} is some
$\bm{A}=(A_-,A_+,\A)$, where $\A$ is a relation from $A_-$ to $A_+$
(that is $A\subset A_-\times A_+$).  The cardinal invariant of the
continuum corresponding to such a triple $\bm{A}$ is defined by:
\[\norm{\bm{A}}\defeq\min\set{\abs{\mathcal F}:
  \mathcal F\text{ is a dominating family with respect to }\bm{A}}\;.
\]
Here, a subset $\mathcal F\subset A_+$ is said to be a
\emph{dominating family} with respect to $\bm{A}$ iff for all $x\in
A_-$ there exists $y\in\mathcal F$ such that $x\mathrel{\A}y$.  The
simplest example is the dominating number $\fr d$, which was described
in the first paragraph.  It is easy to see that $\fr d$ is the
cardinal invariant corresponding to the triple
$(\omega^\omega,\omega^\omega,\leq^*)$, where $\omega^\omega$ denotes
the space of sequences of natural numbers and $\leq^*$ denotes the
eventual domination relation.

We will be interested in the \vojtas\ triples themselves, and not the
corresponding cardinal invariants.  This is really a separate pursuit,
since it is clear that many different \vojtas\ triples may be used to
define the same cardinal number.  Of course many triples do not define
interesting invariants, but our study may be more compelling when they
do.

The natural maps between \vojtas\ triples are (generalized) Tukey
morphisms.  If $\bm{A}$ and $\bm{B}$ are \vojtas\ triples, then a
\emph{Tukey morphism} (or just \emph{morphism}) from $\bm{A}$ to
$\bm{B}$ is a pair of maps:
\[\begin{cases}\phi\colon B_-\to A_-\\\psi\colon A_+\to B_+\end{cases}
\]
such that for all $b_-\in B_-$ and $a_+\in A_+$,
\[\phi(b_-)\mathrel{\A} a_+ \implies b_-\mathrel{\B}\psi(a_+)\;.
\]
In particular, if $(\phi,\psi)$ is a morphism from $\bm{A}$ to
$\bm{B}$ and $\mathcal F$ is a dominating family with respect to
$\bm{A}$, then $\psi(\mathcal F)$ is a dominating family family with
respect to $\bm{B}$.  But the existence of a morphism entails more
than just this.  For instance, the symmetry in the definitions leads
to a notion of duality for triples and morphisms.  If $(\phi,\psi)$ is
a morphism from $\bm{A}$ to $\bm{B}$, then $(\psi,\phi)$ is a morphism
from $\bm{B}^\perp$ to $\bm{A}^\perp$, where $\bm{A}^\perp$ is the
triple defined by $(A_+,A_-,\breve\A)$ and $a\mathrel{\breve\A}a'$ iff
$a'\not\mathrel{\A}a$.

As a consequence of the observation that morphisms send dominating
families to dominating families, it follows that if there is a
morphism from $\bm{A}$ to $\bm{B}$ then
$\norm{\bm{A}}\geq\norm{\bm{B}}$.  Just as cardinal inequalities can
be forced to hold or fail, Tukey morphisms between triples can be
forced to exist or not.  However, assuming some amount of definability
on the triples and morphisms involved, one can use morphisms to
establish absolute cardinal inequalities.

\begin{defn}\
  \label{def:borel}
  \begin{itemize}
  \item The \vojtas\ triple $(A_-,A_+,\A)$ is called \emph{Borel} iff
    $A_-$ and $A_+$ are Borel subsets of Polish spaces, and $\A$ is a
    Borel relation.
  \item If $\bm{A}$ and $\bm{B}$ are Borel, then we say that a
    morphism $(\phi,\psi)$ from $\bm{A}$ to $\bm{B}$ is \emph{Borel}
    if both $\phi$ and $\psi$ are Borel functions.
  \end{itemize}
\end{defn}

If there is a Borel morphism from $\bm{A}$ to $\bm{B}$ then we write
$\bm{A}\geq_{BT}\bm{B}$.  Borel morphisms were initially studied by
Blass \cite{blass-borel}, who first noted that they resolve the
absoluteness problem mentioned above.  Indeed, if there exists a Borel
Tukey morphism from $\bm{A}$ to $\bm{B}$, then the corresponding
cardinal inequality $\norm{\bm{A}}\geq\norm{\bm{B}}$ is absolute to
forcing extensions.  Additionally, Blass was motivated by some more
subtle applications of Borel morphisms.  For instance, consider the
cardinal equalities $\fr r_m=\fr r_n$ for all $m,n$, where here $\fr
r_n$ denotes the \emph{$n$-unsplitting number}: the least cardinality
of a family of reals such that every coloring $c\in m^\omega$ is
almost constant on some member of the family.  In other words, $\fr
r_n$ is defined by the triple
$(n^\omega,[\omega]^\omega,\mathsf{R}_n)$, where
$c\mathrel{\mathsf{R}}_nB$ iff $c$ is almost constant on $B$.  (Thus
$\fr r_2$ is just the usual unsplitting number $\fr r$; see the next
section.)  The proofs of the inequalities $\fr r_m\geq\fr r_n$ for
$2\leq m<n$ can be seen as involving an operation on \vojtas\ triples
called \emph{sequential composition}.  Blass conjectured that for
$2\leq m<n$ the inequality $\fr r_m\geq\fr r_n$ is not witnessed by a
Borel morphism, which we would take to mean that sequential
composition is necessary to prove the inequality.

Since Blass's initial study, however, there have been just a couple of
results on Borel morphisms.  Blass's conjecture concerning $\fr r_n$
was established by Mildenberger, who showed in \cite{mildenberger}
that there are no such Borel morphisms.  Another step was taken in
\cite{paw}, where the authors show that after suitably coding the null
and meager ideals, all of the inequalities in Cicho\'n's diagram are
witnessed by Borel (in fact continuous) morphisms.  On the other hand,
there are a growing number of applications of Borel morphisms
appearing in the literature.  See, for instance, the body of recent
work on parametrized diamond principles initiated in
\cite{parametrized}, or the results in Borel equivalence relations
found in \cite{invar}.

In this paper we wish to renew an interest in the systematic study of
the relationships between cardinal invariants with respect to Borel
morphisms.  We would also like to propose a mild generalization of
this study to certain cardinal invariants which are not definable from
\vojtas\ triples alone.  To see what we mean, consider the almost
disjointness number $\fr a$.  This cardinal is the least size of a
family which is not only dominating with respect to $\not\perp$, but
which is \emph{also} almost disjoint.  Presently, we show how to
handle cardinals which are definable in this more general sense.
(This is motivated in part by Zapletal's \cite{zap}, where such
cardinals are discussed and handled collectively.)

\begin{defn}
  \label{def:cardinal}
  If $\bm{A}$ is a \vojtas\ triple and $P$ is an arbitrary property of
  subsets of $A_+$, then the cardinal invariant of the continuum
  corresponding to $\bm{A}$ and $P$ is
  \[\norm{A}_P\defeq\min\set{\abs{\mathcal F}:\mathcal F\text{
      satisfies property $P$ and is a
      dominating family with respect to $\bm{A}$}}
  \]
\end{defn}

In the case that $P$ is trivial, that is, the cardinal is definable
from a \vojtas\ triple alone, we say that the cardinal is
\emph{simple}.  Again, it makes sense to define Tukey morphisms
between cardinals which are not simple.

\begin{defn}
  \label{def:morphism}
  If $\bm{A}$ and $\bm{B}$ are \vojtas\ triples and $P$ and $Q$ are
  properties, then a morphism from $\bm{A},P$ to $\bm{B},Q$ is a pair
  of maps $\phi\colon B_-\to A_-$ and $\psi\colon A_+\to B_+$
  satisfying:
  \begin{enumerate}
  \item If $\mathcal F$ satisfies property $P$ then $\psi(\mathcal F)$
    satisfies property $Q$, and
  \item $\phi(b_-)\mathrel{\A} a_+ \implies b_-\mathrel{\B}\psi(a_+)$.
  \end{enumerate}
\end{defn}

We are proposing to study Borel Tukey morphisms between a number of
cardinal invariants definable from some $\bm{A}$ and $P$, a more
ambitious plan than that of \cite{blass}.  This extension was proposed
by Coskey and Schneider, who encountered the problem in a slightly
different context \cite{invar}.  By allowing cardinal definitions
where $P$ is nontrivial, we open the door for many important new
cardinals to be compared with respect to Borel morphisms.  For
instance, we can now incorporate into the Borel Tukey order several
new entries from the van~Douwen diagram of combinatorial cardinal
invariants.

We should address the common objection to this programme that the
existence of a Borel morphism is much stronger than is needed to prove
the corresponding cardinal inequality.  To answer this, simply note
that the above-mentioned applications of Borel morphisms to
parametrized diamond principles and Borel equivalence relations cannot
be established on the basis of cardinal inequalities alone.  Thus in
these areas the Borel Tukey order serves as a dictionary of positive
results.  Moreover, results concerning Borel morphisms can have
combinatorial value.  For instance, Mildenberger's discovery that
there is no Borel morphism from $\fr r_m$ to $\fr r_n$ for $m<n$ can
be viewed as new information concerning measurable colorings of
$\omega$ and homogeneity.

This paper is organized as follows.  In the next section, we will
establish a Borel version of van Douwen's diagram.  The third section
is devoted to the proof of just one of the edges in this diagram: the
construction of a morphism from $\fr p$ to $\fr b$.  In the fourth
section, we do for splitting numbers what Blass and Mildenberger did
for unsplitting numbers: we define \emph{$n$-splitting numbers} $\fr
s_n$ and prove an analog of Mildenberger's theorem.  We also consider
the infinite versions $\fr r_\sigma$ and $\fr s_\sigma$ of the
unsplitting and splitting numbers.  In the final section we give a
method for constructing arbitrary patterns in the Borel Tukey order.

\section{A Borel van Douwen diagram}

In this section, we consider the cardinal invariants in van Douwen's
diagram which can be naturally defined using \vojtas\ triples.
Specifically, we consider the cardinal invariants shown in
Figure~\ref{fig:vd}.  The aim is to produce a ``Borel version'' of van
Douwen's diagram, with arrows only in the case that the cardinal
inequalities are witnessed by Borel morphisms.

\begin{figure}[h]
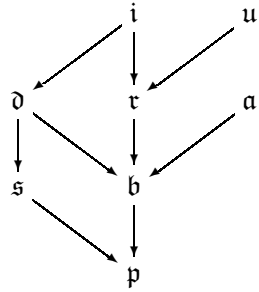

  \caption{Provable size relationships among some combinatorial
    cardinal invariants.  (Here, $\rightarrow$ means
    $\geq$.)\label{fig:vd}}
\begin{displaymath}
\begin{diagram}[h=1.5em,w=2em]
      &       & \fr i &       & \fr u \\
      & \ldTo & \dTo  & \ldTo &      \\
\fr d &       & \fr r &       & \fr a \\
\dTo  & \rdTo & \dTo  & \ldTo &      \\
\fr s &       & \fr b &       &      \\
      & \rdTo & \dTo  &       &      \\
      &       & \fr p &       &      \\
\end{diagram}
\end{displaymath}
\end{figure}

Since each cardinal invariant can be defined by several different
\vojtas\ triples, the answer to the question of whether a cardinal
inequality is witnessed by a Borel morphism will vary depending on the
choice of triples.  We shall take the approach of choosing at our
discretion a particularly natural triple defining to each of the
cardinal invariants in Figure~\ref{fig:vd}.  Afterwords, we can
conflate without confusion the cardinal invariants with their chosen
defining triples.  Thus, the meaning of a shorthand such as $\fr
d\geq_{BT}\fr b$ may be resolved by examining the definitions in
Table~\ref{tab:triples}, below.

Because the table is displayed compactly, it is necessary to explain
some of the terminology.  First, recall that a function $c\in2^\omega$
is said to \emph{split} the infinite set $A$ if $c\restriction A$
takes both values infinitely often.  Next, $IC$ denotes the family of
infinite/co-infinite subsets of $\omega$, and $\perp$ denotes the
relation ``is almost disjoint from''.  Finally, in row $\fr i$, the
property $P(\mathcal F)$ should actually say that ``$\mathcal F$ is
derived from an independent family by taking all intersections of
finitely many sets or their complements''.

\begin{table}[h]
\begin{onehalfspacing}
  \caption{Natural definitions for the cardinal invariants shown in
    Figure~\ref{fig:vd}.\label{tab:triples}}
\begin{tabular}{ccccc}
  cardinal& $A_-$ & $A_+$ & $\A$ & $P=$ ``$\mathcal F$ is \ldots''\\\hline
  $\fr p$ & $[\omega]^\omega$ & $[\omega]^\omega$ & $\not\subset^*$ & centered\\
  $\fr s$ & $[\omega]^\omega$ & $2^\omega$ & is split by & --\\
  $\fr r$ & $2^\omega$ & $[\omega]^\omega$ & does not split & --\\
  $\fr b$ & $\omega^\omega$ & $\omega^\omega$ & $\not\geq^*$ & -- \\
  $\fr d$ & $\omega^\omega$ & $\omega^\omega$ & $\leq^*$ & -- \\
  $\fr a$ & $IC$ & $IC$ & $\not\perp$ & a.\ d., infinite\\
  $\fr i$ & $IC$ & $IC$ & does not split & (see discussion)\\
  $\fr u$ & $[\omega]^\omega$ & $[\omega]^\omega$ & does not split & centered\\
\end{tabular}
\end{onehalfspacing}
\end{table}

We now consider in turn each edge of the diagram.  First, there are a
number of easy answers to be reaped.

\boxhead{$\fr d\rightarrow\fr b$} Since dominating families are
unbounded, this is just a trivial morphism.

\boxhead{$\fr d\rightarrow\fr s$} The classical proof can be seen as a
morphism proof.  See \cite{blass}, Theorem~3.3 and the corresponding
discussion in Section~4 of that article.

\boxhead{$\fr r\rightarrow\fr b$} This is dual to $\fr d\geq\fr s$.

\boxhead{$\fr u\rightarrow\fr r$} The identity maps clearly work.

\boxhead{$\fr i\rightarrow\fr r$} The identity maps clearly work, except that
we technically must define the behavior of $\phi$ on the finite and
cofinite sets.  In fact, if $B$ is cofinite then we can let
$\phi(B)\in[\omega]^\omega$ be arbitrary.

\boxhead{$\fr i\rightarrow\fr d$} There do not exist Borel such maps.
Indeed, suppose that $(\phi,\psi)$ were such a morphism.  Then in
particular, $\phi$ and $\psi$ satisfy
\[\phi(f)\text{ does not split } A\implies f\leq\psi(A)\;.
\]
This implies that $(\phi,\psi)$ is a Borel morphism from $\fr r$ to
$\fr d$.  Now, it is well-known that the inequality $\fr r\geq\fr d$
can be violated in a forcing extension (for instance in the Miller
model), and hence in this extension we have that $\fr
r\not\geq_{BT}\fr d$.  But since both $\fr r$ and $\fr d$ are simple,
the fact that $(\phi,\psi)$ is a morphism from $\fr r$ to $\fr d$
would be preserved to the forcing extension, a contradiction.

\begin{question}
  Is it possible to tinker with the definition of $\fr i$ in such a
  way that this question becomes nontrivial?
\end{question}

\boxhead{$\fr a\rightarrow\fr b$} 
This question is somewhat trivial, since it is easy to see that there
does not exist any morphism (let alone a Borel one) from $\fr a$ to
$\fr b$.

\begin{prop}
  \label{prop:ab}
  There do not exist maps $\phi\colon\omega^\omega\to IC$ and
  $\psi\colon IC\to\omega^\omega$ satisfying
  \[\phi(f)\not\perp A\implies f\not\geq^*\psi(A)\;.
  \]
\end{prop}

\begin{proof}
  Suppose that the $\phi$ and $\psi$ were such maps.  Consider the
  sets $O=$ the odd numbers and $E=$ the even numbers in place of $A$,
  and let $f(n)=\max\set{\psi(O)(n),\psi(E)(n)}$.  Now either
  $\phi(f)\not\perp O$ or else $\phi(f)\not\perp E$, a contradiction
  in either case.
\end{proof}

We find this triviality-of-a-proof unsatisfying, particularly because
it exploits a pair of complementary sets---something never present in
an infinite a.d.\ family.  This would be resolved by a negative answer
to the following, subtler question.

\begin{question}
  \label{q:ab}
  Can there be a pair of maps satisfying the condition of
  Proposition~\ref{prop:ab} just for sets $A$ ranging in some mad
  family of minimal cardinality?
\end{question}

This discussion admits a generalization to cardinals with definitions
similar to that of $\fr a$.  Specifically, for $\mathcal C$ a
collection of filters on $\omega$ let $\fr p_{\mathcal C}$ bet the
cardinal defined by the triple
$([\omega]^\omega,[\omega]^\omega,\not\subset^*)$ together with the
property $P(\mathcal F)=$ ``$\mathcal F$ generates a filter which is
in $\mathcal C$.''  Then $\fr p$ is $\fr p_{\mathcal C}$ where
$\mathcal C$ consists of all filters.  Moreover, it is easy to see
that $\fr a$ is $\fr p_{\mathcal C}$ where $\mathcal C$ consists of
those filters whose dual ideal is generated by an infinite mad family.

\begin{prop}
  For any class of filters $\mathcal C$, we have $\fr p_{\mathcal
    C}\not\geq_{BT}\fr b$ (whether the inequality $\fr p_{\mathcal
    C}\geq\fr b$ is true or false).
\end{prop}


The proof is identical to that of Proposition~\ref{prop:ab} (and we
can ask a version of Question~\ref{q:ab} in this case).  Concerning
morphisms going the other way, the following observation shows that
the problem is closely connected with that of diagonalizing filters.

\begin{prop}
  \label{prop:diag}
  Suppose that it is possible to diagonalize any filter in $\mathcal
  C$ without adding dominating reals.  Then we have $\fr
  b\not\geq_{BT}\fr p_{\mathcal C}$.
\end{prop}

\begin{proof}
  Suppose that there is such a morphism, that is, there exist Borel
  maps $\phi\colon[\omega]^\omega\to\omega^\omega$ and
  $\psi\colon\omega^\omega\to[\omega]^\omega$ satisfying
  \begin{enumerate}
  \item $\psi(\omega^\omega)$ generates a filter in $\mathcal C$, and
  \item $A\subset^*\psi(f)\implies f\leq^*\phi(A)$.
  \end{enumerate}
  Then by our assumption, it is possible to force to add a
  pseudo-intersection $\dot A$ of $\psi(\omega^\omega)$ without adding
  dominating reals.  Thus there exists $f\in\omega^\omega\cap V$ such
  that $f\not\leq^*\phi(\dot A)$.  Since $\phi$ is Borel, in the
  extension we have that for all $x\in[\omega]^\omega$:
  \[x\subset^*\psi(f)\implies f\leq^*\phi(x)\;.
  \]
  Plugging in $x=\dot A$ yields an immediate contradiction.
\end{proof}

Of course, in the case of $\fr p_{\mathcal C}=\fr a$, we already know
that there is no Borel morphism from $\fr b$ to $\fr a$ (since the
inequality $\fr b\geq\fr a$ can be forced to fail).  It would be
interesting to give a proof of this using a diagonalization argument.
Notice also that the proof of Proposition~\ref{prop:diag} shows
outright that $\phi$ cannot be Borel.  It would be nice to find a
condition which implies $\psi$ cannot be Borel.

\boxhead{$\fr a\rightarrow\fr p$} Since we established that $\fr
a\not\geq_{BT}\fr b$, it is natural to ask whether we even have $\fr
a\geq_{BT}\fr p$.  Indeed, this is the case, since the maps
$\psi(A)=\omega\smallsetminus A$ and $\phi=\id$ satisfy the
requirements:
\begin{enumerate}
\item if $\mathcal F$ is a.\ d. and infinite then $\psi(\mathcal F)$
  is centered, and
\item $\phi(A)\not\perp B\implies A\not\subset^*\psi(B)$.
\end{enumerate}
Thus $\fr a$ hasn't fallen off of the diagram!

\boxhead{$\fr s\rightarrow\fr p$} The following result shows that in
fact $\fr s\not\geq_{BT}\fr p$.

\begin{thm}
  \label{thm:sp}
  Suppose that $\phi,\psi\colon[\omega]^\omega\to[\omega]^\omega$ are
  maps satisfying:
  \begin{enumerate}
  \item $\psi([\omega]^\omega)$ is centered, and
  \item $A\subset^*\psi(B)\implies B\text{ does not split }\phi(A)$.
  \end{enumerate}
  Then $\phi$ and $\psi$ cannot both be Borel.
\end{thm}

Coskey and Schneider have previously established this result under the
additional assumption that $\phi$ is $E_0$-invariant (\emph{i.e.},
$A=^*A'$ iff $\phi(A)=^*\phi(A')$).  However, that fact is now
superseded by the following shorter and stronger argument, which was
pointed out to us by Dilip Raghavan.

\begin{proof}[Proof of Theorem~\ref{thm:sp}]
  Suppose that $(\phi,\psi)$ are Borel functions satisfying (a) and
  (b).  Letting $\mathcal F$ denote the filter generated by
  $\psi([\omega]^\omega)$, we use the (relativized) Mathias forcing to
  add a pseudo-intersection $\dot A$ for $\mathcal F$.  This forcing
  is always ccc, and since $\mathcal F$ is analytic, the forcing is
  Suslin as well (see \cite[Definition~3.6.1]{bart}).  It follows from
  \cite[Lemma~3.6.24]{bart} that the ground model is a splitting
  family in the forcing extension.  In particular, there exists
  $B\in[\omega]^\omega\cap V$ such that $B$ splits $\phi(\dot A)$.  In
  the ground model, we apply (b) to obtain:
  \[(\forall x\in[\omega]^\omega)\;\;
  x\subset^*\psi(B)\implies B\text{ does not split }\phi(x)
  \]
  Since $\phi$ is Borel, the same sentence holds in the extension.  It
  follows that $B$ does not split $\phi(\dot A)$, which is a
  contradiction.
\end{proof}

We remark that the argument of Theorem~\ref{thm:sp} also shows that
there is no Borel morphism from $\fr s_\sigma$ to $\fr p$ (for the
definition of $\fr s_\sigma$, see the later section on splitting).
We leave open the following question:

\begin{question}
  Does there exist a morphism $(\phi,\psi)$ from $\fr s$ to $\fr p$
  such that just one of the maps is Borel?
\end{question}

\boxhead{$\fr b\rightarrow\fr p$} It is the case that $\fr
b\geq_{BT}\fr p$.  Since the construction is fairly involved, we shall
give the proof its own section, below.

\vspace{6pt} To complete our discussion of van~Douwen's diagram, we
finally verify that whenever an edge does \emph{not} appear in
Figure~\ref{fig:vd}, then there is not a Borel morphism either.  Most
of this verification is routine, because it is already known that any
cardinal inequality not shown in Figure~\ref{fig:vd} can be violated
by forcing.  Hence, if there is no edge between \emph{simple}
invariants $\norm{\bm{A}}$ and $\norm{\bm{B}}$ in Figure~\ref{fig:vd},
then we automatically obtain $\bm{A}\not\geq_{BT}\bm{B}$.

Even when just one of the invariants is involved is simple, a forcing
argument will work.  Indeed, if there is a Borel morphism from
$\bm{A},P$ to $\bm{B}$, then the condition in
Definition~\ref{def:morphism}(a) is trivial, and so it is preserved to
forcing extensions.  On the other hand, if there is a Borel morphism
from $\bm{A}$ to $\bm{B},Q$, and property $Q$ is closed downward, then
the condition in Definition~\ref{def:morphism}(a) amounts to saying
that all of $\im(\psi)$ has property $Q$.  Since all of the cardinals
we are considering are defined by a property $Q$ which is closed
downward and very low in complexity, this will again be preserved to
forcing extensions.

Hence, we need only handle the inequalities between cardinals which
are both not simple.  This is done in the next result.

\begin{prop}
  The invariants $\fr i$, $\fr u$ and $\fr a$ are incomparable with
  respect to $\geq_{BT}$.
\end{prop}

\begin{proof}
  Referring to the definitions of $\fr i$ and $\fr u$, it is clear
  that if we had either $\fr a\geq_{BT}\fr i$ or $\fr a\geq_{BT}\fr
  u$, then we would have from $\fr a\geq_{BT}\fr r$.  But now the
  inequality $\fr a\geq\fr r$ can be violated by forcing and $\fr r$
  is simple, so we can use the argument above.

  The rest of the cases are similar.  If we had $\fr u\geq_{BT}\fr i$
  then we would also have $\fr r\geq_{BT}\fr i$; if we had $\fr
  u\geq_{BT}\fr a$ then we would also have $\fr r\geq_{BT}\fr a$; if
  we had $\fr i\geq_{BT}\fr u$ then we would have $\fr r\geq_{BT}\fr
  u$; if we had $\fr i\geq_{BT}\fr a$ then we would have $\fr
  r\geq_{BT}\fr a$.  In all four of these cases, the argument above
  applies.
\end{proof}

The results of this section are summarized in Figure~\ref{fig:bvd}.

\begin{figure}[h]
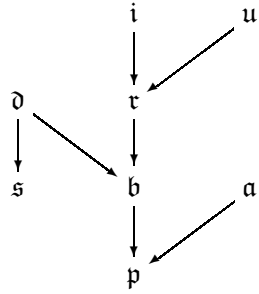

  \caption{Borel Tukey morphisms among some combinatorial cardinal
    invariants.  (Here, $\rightarrow$ means
    $\geq_{BT}$.)\label{fig:bvd}}
\begin{displaymath}
\begin{diagram}[h=1.5em,w=2em]
      &       & \fr i &       & \fr u \\
      &       & \dTo  & \ldTo &       \\
\fr d &       & \fr r &       &       \\
\dTo  & \rdTo & \dTo  &       &       \\
\fr s &       & \fr b &       & \fr a \\
      &       & \dTo  & \ldTo &       \\
      &       & \fr p &       &       \\
\end{diagram}
\end{displaymath}
\end{figure}

\begin{question}
  Is there an interesting alternative set of definitions of these
  invariants for which the Borel morphisms faithfully reflect all of
  the inequalities in van~Douwen's diagram?
\end{question}

For instance, we know that there is no Borel morphism from $\fr i$ to
$\fr d$ as we have defined them.  But it is worth mentioning that if
$\phi$ and $\psi$ are the maps constructed in the
Theorem~\ref{thm:bp}, then property \ref{thm:bp}(b) comes very close
to giving the condition needed for a morphism from $\fr i$ to $\fr d$
(with the roles of $\phi$ and $\psi$ interchanged).  Hence it may be
possible to give a new proof that $\fr i\geq\fr d$ by slightly
modifying the triple for $\fr i$ and the construction in
Theorem~\ref{thm:bp}.

\boxhead{$\fr t$} No discussion involving $\fr p$ would be complete
without mentioning the tower number, $\fr t$.  This cardinal is
defined by the same triple as $\fr p$, together with the property
$P(\mathcal F)=$ ``$\mathcal F$ is linearly ordered''.  Clearly $\fr
t\geq_{BT}\fr p$, but it has only recently been shown by Malliaris and
Shelah \cite{malliaris} that $\fr p\geq\fr t$.  Thus, it is desirable
to verify that the latter inequality doesn't have a Borel proof.

\begin{prop}
  We have $\fr p\not\geq_{BT}\fr t$.
\end{prop}

\begin{proof}
  Suppose towards a contradiction that $(\phi,\psi)$ satisfy
  \begin{enumerate}
  \item if $\mathcal F$ is centered then $\psi(\mathcal F)$ is
    linearly ordered, and
  \item $\phi(x)\not\subset^*y\implies x\not\subset^*\psi(y)$.
  \end{enumerate}
  Let $A,B,C$ be any three infinite sets with empty intersection but
  such that any two have infinite intersection.  Then applying (a) to
  each pair $\set{A,B},\set{A,C},\set{B,C}$ we conclude that
  $\psi(\set{A,B,C})$ is linearly ordered.  Thus $\psi(\set{A,B,C})$
  has infinite intersection, and using (b), it follows that
  $\set{A,B,C}$ does too.  This contradicts the choice of $A,B,C$.
\end{proof}

Once again this result is rather trivial, so it would be interesting
to rework the question to yield a more poignant theorem.  Moreover, we
were unable to include $\fr t$ in Figure~\ref{fig:bvd} since we do not
know its Borel relationship with the other invariants.  Thus we are
left with the following question:

\begin{question}
  Can the result of the next section be improved to show that $\fr
  b\geq_{BT}\fr t$?
\end{question}

\section{A Borel morphism from $\fr b$ to $\fr p$}

Although the simplest proof that $\fr b\geq\fr p$ does not give a
Borel morphism, the following result establishes that it is indeed the
case that that $\fr b\geq_{BT}\fr p$.

\begin{thm}
  \label{thm:bp}
  There exists a continuous map
  $\psi\colon\omega^\omega\to[\omega]^\omega$ and a Borel map
  $\phi\colon[\omega]^\omega\to\omega^\omega$ satisfying:
  \begin{enumerate}
  \item $\psi(\omega^\omega)$ is centered, and
  \item $A\subset^*\psi(f)\implies f\leq^*\phi(A)$.
  \end{enumerate}
\end{thm}

\begin{proof}[Proof]
  The outline of the proof is as follows.  We will construct a
  continuous map $\psi$ which satisfies (a), and has the additional
  property that for every $\leq^*$-unbounded subset
  $S\subset\omega^\omega$, the image $\psi(S)$ does not have a
  pseudo-intersection.  In particular, we will have:
  \begin{itemize}
  \item[($\star$)] for each $A\in[\omega]^\omega$, the set
    $C_A\defeq\set{f\in\omega^\omega\mid A\subset^*\psi(f)}$ is
    $\leq^*$-bounded.
  \end{itemize}
  Letting $\phi(A)$ be such a bound, it is easy to see that
  $(\phi,\psi)$ satisfy property (b).  We will show moreover that
  bounds $\phi(A)$ can be chosen in a Borel fashion.

  We now begin the construction of $\psi$.  Let $n\in\NN$, and let
  $T_n$ denote the tree which is $\omega$-branching for the first $n$
  levels, and binary branching afterward.

  \begin{claim}
    \label{claim:step1}
    There exists a continuous function
    $\psi_n\colon[T_n]\to[\omega]^\omega$ with the properties:
    \begin{enumerate}
    \item[(i)] if $f_1,\ldots,f_n\in[T_n]$ then
      $\psi_n(f_1)\cap\cdots\cap\psi_n(f_n)$ is infinite,
    \item[(ii)] if $f_1,\ldots,f_{n+1}\in[T_n]$ are all distinct, then
      $\psi_n(f_1)\cap\cdots\cap\psi_n(f_{n+1})$ is finite, and
    \item[(iii)] if $f_1,\ldots,f_{n+1}\in[T_n]$ and
      $f_1\seqrestriction n,\ldots,f_{n+1}\seqrestriction n$ are all
      distinct, then we even have that
      $\psi_n(f_1)\cap\cdots\cap\psi_n(f_{n+1})=\emptyset$.
    \end{enumerate}
  \end{claim}

  \begin{claimproof}
    It will be more convenient to construct the map $\psi_n$ from
    $[T_n]$ into the set $[\Omega]^\omega$, where
    \[\Omega=\set{(t_1,\ldots,t_n)\in(T_n)^n\mid (\exists l>n)\;(\forall
      i)\;\lev(t_i)=l\;\&\;\code{(t_1\seqrestriction n,\ldots
        t_n\seqrestriction n)}<l}\;.
    \]
    Here, $\code{\cdot}$ denotes any fixed bijection from
    $(\omega^n)^n\to\omega$.  Now, we simply define
    \[\psi_n(f)=\set{(t_1,\ldots,t_n)\in\Omega\mid (\exists
      i)\;t_i\subset f}\;.
    \]
    To see (i), let $f_1,\ldots,f_n\in[T_n]$ be given.  Then for $l$
    large enough we will have that $(f_1\seqrestriction l,\ldots
    f_n\seqrestriction l)\in\Omega$, and moreover these sequences will
    lie in $\psi_n(f_1)\cap\cdots\cap\psi_n(f_n)$.

    Assertion (iii) just follows from the pigeon-hole principle: no
    sequence of length $n$ will suffice to match $n+1$-many initial
    segments.

    For assertion (ii), if $f_1,\ldots,f_{n+1}$ are all distinct, then
    there exists a level $l$ such that $f_1\seqrestriction
    l,\ldots,f_{n+1}\seqrestriction l$ are all distinct.  Using the
    same pigeon-hole argument as above, any element $(t_1,\ldots,t_n)$
    of $\psi_n(f_1)\cap\cdots\cap\psi_n(f_{n+1})$ must lie at some
    level $l'<l$.  But there are only finitely many such
    $(t_1,\ldots,t_n)$, since we require $\code{(t_1\seqrestriction
      n,\ldots t_n\seqrestriction n)}<l'$ and $T_n$ is finitely
    branching after level $n$.
  \end{claimproof}
  
  To define $\psi$, we simply ``glue together'' all of the $\psi_n$.
  More precisely, for each $n$, we regard $\omega^\omega$ as a subset
  of $[T_n]$ and therefore think of $\psi_n$ as a function from
  $\omega^\omega$ into $[\omega]^\omega$.  Using this tacitly, we let
  $\psi$ be the function from
  $\omega^\omega\to[\omega\times\omega]^\omega$ defined by placing
  $\psi_n$ on the $n\th$ column.  (Even more precisely, for each $n$,
  we fix an embedding from $\omega^{<\omega}$ into $T_n$ which is
  equal to the identity on the first $n$ levels.  Letting
  $\iota_n\colon\omega^\omega\to[T_n]$ denote the induced injection,
  this allows us to replace $\psi_n$ with $\psi_n\circ\iota_n$
  without harming appeals to Claim~\ref{claim:step1}(iii).  We then
  let
  \[\psi(f)=\set{(n,m)\mid m\in\psi_n(\iota_n(f))}\;.
  \]
  Of course, we may also use a pairing function to think of $\psi$ as
  a function into $[\omega]^\omega$.  In our arguments, we will freely
  elide the use of $\iota_n$ and this pairing function.)

  With this definition, it is clear that $\psi(\omega^\omega)$ is
  centered.  Indeed, given a sequence
  $f_1,\ldots,f_n\in\omega^\omega$, we have from
  Claim~\ref{claim:step1}(i) that
  $\psi_n(f_1)\cap\cdots\cap\psi_n(f_n)$ is infinite, and hence so is
  $\psi(f_1)\cap\cdots\cap\psi(f_n)$.  To get property ($\star$), we
  use the following auxiliary claim.

  \begin{claim}
    \label{claim:step2}
    If $S\subset\omega^\omega$ is $\leq$-unbounded then $\bigcap_{f\in
      S}\psi(f)$ is finite.
  \end{claim}

  \begin{claimproof}
    First note that if $S$ is infinite, then for each $n$ we have that
    $\bigcap_{f\in S}\psi_n(f)$ is finite by
    Claim~\ref{claim:step1}(ii).  Hence $\bigcap_{f\in S}\psi(f)$
    meets each column of $\omega\times\omega$ in a finite set.  Now,
    if additionally $S\subset\omega^\omega$ is $\leq$-unbounded, then
    it is not hard to see that there exists a level $l$ and elements
    $f_1,f_2,\ldots\in S$ such that $f_1\seqrestriction
    l,f_2\seqrestriction l,\ldots$ are all distinct.  It follows from
    Claim~\ref{claim:step1}(iii) that for every $n>l$, we have
    $\bigcap_i\psi_n(f_i)=\emptyset$.  Hence $\bigcap_{f\in S}\psi(f)$
    only meets finitely many columns of $\omega\times\omega$.  Putting
    these together, we can conclude that $\bigcap_{f\in S}\psi(f)$ is
    finite.
  \end{claimproof}

  We can now conclude that that $\psi$ satisfies property ($\star$).
  Indeed, for all $A\in[\omega]^\omega$, Claim~\ref{claim:step2}
  implies that for each $n$ the set
  \[C_{A,n}\defeq\set{f\in\omega^\omega\mid A\smallsetminus
    n\subset\psi(f)}
  \]
  is $\leq$-bounded, and hence that $C_A$ is $\leq^*$-bounded.
  However, to show that a bound $\phi(A)$ can be obtained from $A$ in
  a Borel fashion, we need one more claim.  In the following result,
  we will let $\mathcal K(\omega^\omega)$ denote the space of compact
  subsets of $\omega^\omega$ endowed with its usual hyperspace
  topology (called the Vietoris topology).  Note that this space
  includes all of the $C_{A,n}$ because they are closed and
  $\leq$-bounded.

  \begin{claim}
    For all $n$, the function $[\omega]^\omega\to\mathcal
    K(\omega^\omega)$ defined by $A\mapsto C_{A,n}$ is Borel.
  \end{claim}
  
  \begin{claimproof}
%
%
%
    Since the map $A\mapsto A\smallsetminus n$ is continuous, it is
    enough to treat the case $n=0$, that is, to show that the map
    \[A\mapsto\set{f\in\omega^\omega\mid A\subset\psi(f)}
    \]
    is Borel.  By \cite[Theorem~28.8]{kechris}, if $X$ and $Y$ are
    Polish then $\alpha\colon X\to \mathcal K(Y)$ is Borel iff the
    relation $\set{(x,y):y\in \alpha(x)}$ is Borel.  Thus, to
    establish the claim, we need only verify that
    \[\set{(A,f)\mid A\subset\psi(f)}
    \]
    is a Borel subset of $[\omega]^\omega\times\omega^\omega$.  But
    this follows easily from Suslin's theorem, because
    $A\subset\psi(f)$ if and only if there exists
    $B\in[\omega]^\omega$ such that $B=\psi(f)$ and $A\subset B$, and
    also $A\subset\psi(f)$ if and only if for all
    $B\in[\omega]^\omega$ if $B=\psi(f)$ then $A\subset B$.
  \end{claimproof}

  With this in hand, we can define $\phi(A)$ as follows.  For all $n$,
  since $C_{A,n}$ is closed and $\leq$-bounded, we can find its least
  upper bound $b_n$.  (It is an easy exercise to check that the map
  $\mathcal K(\omega^\omega)\to\omega^\omega$ which sends a bounded
  set to its least upper bound is continuous).  Then, simply
  diagonalize to find $\phi(A)$ such that for all $n$,
  $b_n\leq^*\phi(A)$.  This concludes the proof of
  Theorem~\ref{thm:bp}
\end{proof}

We remark that the result cannot be improved to get $\phi$ continuous
too.  Indeed, if $\psi$ is even Borel then $\psi(\omega^\omega)$
generates a Baire measurable filter $\mathcal F$.  A well-known result
of Talagrand \cite{talagrand} and Jalali-Naini \cite{naini} implies
that there exists a partition of $\omega$ into finite intervals
$(J_n)$ such that for all $F\in\mathcal F$ and almost all $n$, $F\cap
J_n\neq\emptyset$.  Now, let $\mathcal A$ be the collection of almost
transversals for $(J_n)$, \emph{i.e.}, sets $A$ such that $\abs{A\cap
  J_n}\leq1$ for all $n$ and $\abs{A\cap J_n}=1$ for all but finitely
many $n$.  Then $\mathcal A$ is $\sigma$-compact, and if $\phi$ is
continuous then $\phi(\mathcal A)$ is $\sigma$-compact as well.  It
follows that $\phi(\mathcal A)$ is $\leq^*$-bounded, say by $f$.  Now,
$\psi(f)\in\mathcal F$ clearly contains some subset $A$ such that
$A\in\mathcal A$.  But then property (b) implies that
$f\leq^*\phi(A)$, and this is a contradiction.

\section{Splitting and unsplitting}

In this section, we consider the so-called $\sigma$-splitting number
$\fr s_\sigma$ and the $\sigma$-unsplitting number $\fr r_\sigma$.
These cardinals are closely related to $\fr s$ and $\fr r$: we have
both $\fr s_\sigma\geq\fr s$ and $\fr r_\sigma\geq\fr r$, and we don't
know whether either of the reverse inequalities are theorems of \ZFC.
We will show in each case that the unknown inequalities do not hold in
the Borel Tukey order.  For $\fr r_\sigma$, this result follows
trivially from Mildenberger's result concerning the cardinals $\fr
r_n$ which was mentioned in the introduction.  For $\fr s_\sigma$, we
will follow a similar strategy and define a family of cardinals $\fr
s_n$ which in some sense approximate $\fr s_\sigma$.

\boxhead{$\fr r_\sigma$} The \emph{$\sigma$-unsplitting} number, $\fr
r_\sigma$, is defined to be the least cardinality of a family of reals
such that no countable subset of $2^\omega$ suffices to split them
all.  In other words, it is defined by the triple
$((2^\omega)^\omega,[\omega]^\omega,\mathsf{R}_\sigma)$ where
$\seq{c_n}\mathrel{\mathsf{R}}_\sigma B$ iff for all $n$, $c_n$ is
almost constant on $B$.  It is clear that $\fr r_\sigma\geq\fr r$, in
fact the trivial maps $\phi(c)=\seq{c}$ (the constant sequence) and
$\psi=\id$ give a morphism.  On the other hand, it is an important
open question whether the reverse inequality $\fr r\geq\fr r_\sigma$
holds.

\begin{cor}[of Mildenberger]
  \label{cor:rrsigma}
  We have $\fr r\not\geq_{BT}\fr r_\sigma$.
\end{cor}

\begin{proof}
  We have just seen that $\fr r_\sigma\geq_{BT}\fr r$, and we can
  similarly show that $\fr r_\sigma\geq_{BT}\fr r_n$ for $n>2$.
  Indeed, we require maps $\phi\colon n^\omega\to(2^\omega)^\omega$
  and $\psi\colon[\omega]^\omega\to[\omega]^\omega$ such that:
  \[\phi(c)(i)\text{ almost constant on }B\;(\forall i)
  \implies c\text{ almost constant on }\psi(B)\;.
  \]
  For this, we simply take $\phi(c)(i)$ to be the $2$-coloring which
  assigns to $k$ the $i\th$ bit of $c(k)$, and $\psi=\id$ as before.
  But now if we had $\fr r\geq_{BT}\fr r_\sigma$, then since $\fr
  r_\sigma\geq_{BT}\fr r_n$ we would have $\fr r\geq_{BT}\fr r_n$,
  contradicting the result of Mildenberger that there are no Borel
  morphisms from $\mathfrak r_m$ to $\mathfrak r_n$ for $m<n$.
\end{proof}

It should be noted that Spinas \cite{spinas} has strengthened
Mildenberger's result, showing that if $(\phi,\psi)$ is a morphism
from $\fr r_m$ to $\fr r_n$ ($m<n$) then $\phi$ is not Borel.  This
gives us the analogous strengthening in the case of morphisms from
$\fr r$ to $\fr r_\sigma$.

\boxhead{$\fr s_\sigma$} The \emph{$\sigma$-splitting number} $\fr
s_\sigma$ is the least cardinality of a family $\mathcal F$ such that
for any $A_1,A_2,\ldots\in[\omega]^\omega$ there exists $F\in\mathcal
F$ which splits them all.  In other words, it is defined by the triple
$((2^\omega)^\omega,[\omega]^\omega,\mathsf{S}_\sigma)$ where
$\seq{A_n}\mathrel{\mathsf{S}}_\sigma B$ iff for all $n$, $A_n$ is
split by $B$.  It is easy to see that $\fr s\leq\fr s_\sigma\leq\fr
d$, and both of these inequalities are witnessed by Borel morphisms.
(For the first inequality a trivial morphism works, and for the second
inequality the usual proof gives a morphism).  It is not known whether
$\fr s\geq\fr s_\sigma$ is a true inequality, and so it is natural to
ask for verification that there is no Borel proof of it.

Emulating the example of $\fr r$, $\fr r_n$ and $\fr r_\sigma$, we can
similarly define the cardinal $\fr s_n$ to be the least cardinality of
an \emph{$n$-splitting family}, that is, a family $\mathcal F$ such
that given any $A_1,\ldots,A_n$ there exists $B\in\mathcal F$ which
splits them all.  In other words, $\fr s_n$ is defined by the triple
$(([\omega]^\omega)^n,2^\omega,\mathsf{S}_n)$, where
$\seq{A_1,\ldots,A_n}\mathrel{\mathsf{S}}_nB$ iff for all $i$, $A_i$
is split by $B$.  Thus $\fr s$ and $\fr s_1$ are exactly the same by
definition.  It is clear that $\fr s_n=\fr s_m$ and $\fr
s_n\geq_{BT}\fr s_m$ for $m<n$.  We have the following analog of
Mildenberger's result.

\begin{thm}
  \label{thm:sn}
  $\fr s_n\not\geq_{BT}\fr s_{n+1}$ for all $n$.
\end{thm}

\begin{proof}
  We prove the stronger fact that there is no Baire measurable
  function $\psi$ which carries $n$-splitting families to
  $n+1$-splitting families.  For this, we will first focus on the
  proof that there is no Baire measurable function $\psi$ which
  carries ($1$-)splitting families to $2$-splitting families.
  Afterwards, we will show how to modify the argument in the case when
  $n>1$.

  Suppose, towards a contradiction, that $\psi$ carries splitting
  families to $2$-splitting families and that $\psi$ is continuous on
  a comeager set $G$.  Let $O_n$ be a decreasing family of dense open
  sets such that $\bigcap O_n\subset G$.  We will construct a
  partition $(I_k)$ of $\omega$ into finite intervals, a sequence of
  distinct integers $a_k$, and a family of sequences
  $\set{\theta(s)\in 2^{I_k}: s\in2^{I_{<k}}}$ (where $I_{<k}$ denotes
  $\bigcup_{j<k}I_j$).  In our construction, we shall ensure that for
  each $s\in2^{I_{<k}}$ the following are satisfied:
  \begin{enumerate}
  \item $N_{s\cup\theta(s)}\subset O_k$, and
  \item for all $c\in N_{s\cup\theta(s)}\cap G$ we have $\psi(c)(a_k)=1$.
  \end{enumerate}
  Here as usual the notation $N_s$ means the basic open neighborhood
  of $2^\omega$ corresponding to the sequence $s$.  Borrowing some
  terminology from \cite{blass-evasion}, let us say that $\theta$
  \emph{predicts} $c\in2^\omega$ at level $k$ iff
  $\theta(c\seqrestriction{I_{<k}})=c\seqrestriction{I_k}$.  Thus,
  condition (a) implies that if $\theta$ predicts $c$ at infinitely
  many levels, then $c$ lies in $G$.
  
  Admitting the construction, we consider the families
  \begin{align*}
    \mathcal S_{even}&=\set{c\in2^\omega:\theta\text{ predicts $c$ at
        all even levels}}\\
    \mathcal S_{odd} &=\set{c\in2^\omega:\theta\text{ predicts $c$ at
        all odd levels}}\;.
  \end{align*}
  Then $\mathcal S_{even}$ can split any subset of
  $\bigcup_jI_{2j+1}$, and $\mathcal S_{odd}$ can split any subset of
  $\bigcup_jI_{2j}$.  It follows easily that $\mathcal
  S_{even}\cup\mathcal S_{odd}$ is a splitting family.  On the other
  hand, using (b), we have that $\psi(\mathcal S_{even})$ cannot split
  the set $\set{a_{2j}:j\in\omega}$, and $\psi(\mathcal S_{odd})$
  cannot split the set $\set{a_{2j+1}:j\in\omega}$.  We therefore
  conclude that while $\mathcal S_{even}\cup\mathcal S_{odd}$ is a
  splitting family, $\psi(\mathcal S_{even}\cup\mathcal S_{odd})$ is
  not a $2$-splitting family, a contradiction.

  We now turn to the construction.  Suppose that $I_j$, $a_j$ and
  $\theta(s)$ have been defined for $j<k$ and $s\in 2^{I_{<j}}$.  For
  each $s\in2^{I_{<k}}$, we first use the fact that $O_k$ is dense
  open to find a $t(s)$ such that $N_{s\cup t(s)}\subset O_k$.  This
  will imply that after the construction, (a) will be satisfied.
  Next, roughly speaking, we will use the continuity of $\psi$ on $G$
  to find $\theta(s)\supset t(s)$ which decides certain values of
  $\psi(c)(m)$ for $c\in N_{s\cup\theta(s)}\cap G$.  To satisfy (b),
  we just need to ensure that we can find some $m$ where this value is
  always decided to be $1$.

  \begin{claim}
    For each $s\in2^{I_{<k}}$, there are only finitely many $m$ such
    that for all $c\in N_{s\cup t(s)}\cap G$ we have $\psi(c)(m)=0$.
  \end{claim}
  
  \begin{claimproof}
    Otherwise there would be an infinite subset $Z\subset\omega$ such
    that for all $c\in N_{s\cup t(s)}\cap G$ we have
    $\psi(c)\restriction Z=0$.  But this implies that $\psi(N_{s\cup
      t(s)}\cap G)$ is not a splitting family, which is a
    contradiction because $N_{s\cup t(s)}\cap G$ is nonmeager and
    hence splitting, and $\psi$ takes splitting families to splitting
    families.
  \end{claimproof}
  
  Now, we can choose $a_k>a_{k-1}$ so large that for all $s\in
  2^{I_{<k}}$ there exists $c_s\in N_{s\cup t(s)}\cap G$ such that
  $\psi(c_s)(a_k)=1$.  Using the continuity of $\psi$, we can choose
  $\theta(s)\subset c_s$ extending $t(s)$ so that for all $c\in
  N_{s\cup\theta(s)}\cap G$ we have $\psi(c)(a_k)=1$.  Lengthening
  $\theta(s)$ if necessary, we can suppose that they all have the same
  domain, which we take for $I_k$.  This completes the construction,
  and the proof in the case when $n=1$.

  Finally, we briefly show how to change the argument when $n>1$.
  Rather than defining $\mathcal S_{even}$ and $\mathcal S_{odd}$, we
  simply define $\mathcal S_r$ to be the set of $c$ such that $\theta$
  predicts $c$ at all levels which are congruent to $r$ modulo $n$.
  Then it is not hard to verify that $\mathcal
  S_0\cup\cdots\cup\mathcal S_{n-1}$ is an $n-1$-splitting family.
  But no element of $\psi(\mathcal S_0\cup\cdots\cup\mathcal S_{n-1})$
  can simultaneously split all of the sets
  $A_r=\set{a_{nj+r}:j\in\omega}$ for $r<n$.
\end{proof}

\begin{cor}
  \label{cor:ssigma}
  We have $\fr s\not\geq_{BT}\fr s_\sigma$.
\end{cor}

\begin{proof}
  This is just the same simple argument of
  Corollary~\ref{cor:rrsigma}.  Suppose there were a Borel morphism
  from $\fr s$ to $\fr s_\sigma$.  Then composing it with a (trivial)
  morphism from $\fr s_\sigma$ to $\fr s_2$ we would obtain one from
  $\fr s$ to $\fr s_2$, contradicting Theorem~\ref{thm:sn}.
\end{proof}

The argument of Theorem~\ref{thm:sn} can also be used to separate
(arbitrary) finite splitting from infinite splitting.  That is, if we
define the cardinal $\fr s_{<\omega}$ by the relation
$\mathsf{S}_{<\omega}=\bigcup_{n\in\omega}\mathsf{S}_n$, then we have
the following corollary to the proof of Theorem~\ref{thm:sn}.

\begin{cor}
  We have $\fr s_{<\omega}\not\geq_{BT}\fr s_\sigma$.
\end{cor}

\begin{proof}
  Suppose towards a contradiction that $\psi$ is a Borel map which
  carries finitely splitting families to infinitely splitting
  families, and construct $I_k$, $\theta$, and $a_k$ as before.  Now,
  we simply put together all of the partial splitting families used in
  the proof of Theorem~\ref{thm:sn}.  Namely, let $\mathcal S_{n,r}$
  denote the set of all $c$ such that $\theta$ predicts $c$ at all
  levels which are congruent to $r$ modulo $n$.  Then $\bigcup\mathcal
  S_{n,r}$ is clearly $n$-splitting for all $n$, but no element of
  $\psi(\bigcup\mathcal S_{n,r})$ can ever simultaneously split all of
  the sets $A_r=\set{a_{nj+r}:j\in\omega}$ for $n\in\omega$ and $r<n$.
\end{proof}

\section{A sea of splitting numbers}

In this last section we describe a family \vojtas\ triples of size
continuum, each of which describes the usual splitting number as a
cardinal invariant, but which are Borel Tukey inequivalent.  Similar
results have appeared before; it is known that there is a continuum of
triples which are incomparable even up to ordinary Tukey equivalence.
Our result gives a method of producing essentially arbitrary patterns
in the $\leq_{BT}$ ordering.

To begin, we will need to extend the methods of the previous section
to produce antichains as well as chains.  Building on our earlier
notation, for $m\leq n$ we let $\fr s_{n,m}$ denote the least
cardinality of an \emph{$n,m$-splitting family}: that is, an $\mathcal
F\subset2^\omega$ such that for any sequence $A_1,\ldots, A_n$ of
infinite subsets of $\omega$ there exists $B\in\mathcal F$ which
splits at least $m$ of them.  Thus $\fr s_{n,m}$ is defined by a
triple $(([\omega]^\omega)^n,2^\omega,\mathsf{S}_{n,m})$ where
$\mathsf{S}_{n,m}$ denotes the relation ``at least $m$ of which are
split by''.  Again, all of the cardinals $\fr s_{n,m}$ are equal to
$\fr s$, but it is not immediately clear which pairs are related by a
Borel Tukey map.  The following result computes precisely when this is
the case.

\begin{prop}
  \label{prop:compute}
  Let $m\leq n$ and $m'\leq n'$.
  \begin{itemize}
  \item If $m<m'$ then there is not a Borel Tukey morphism from
    $\fr s_{n,m}$ to $\fr s_{n',m'}$.
  \item If $m\geq m'$ then there is a Borel Tukey morphism from $\fr
    s_{n,m}$ to $\fr s_{n',m'}$ if and only if the following holds:
    \begin{equation}
      \label{eq:buckets}
      \left\lfloor\frac{n}{n'}\right\rfloor(m'-1)+\min\left(r,m'-1\right)<m\;,
    \end{equation}
    where $r$ denotes the remainder upon dividing $n$ by $n'$.
  \end{itemize}
\end{prop}

The combinatorial condition in Equation~\eqref{eq:buckets} means: if
you spread $n$ balls evenly over $n'$ ordered buckets (with the
remainder spread over the left-most buckets), then among the first
$m'-1$ buckets there are fewer than $m$ balls.  For a diagram
depicting this scenario see Figure~\ref{fig:buckets}.

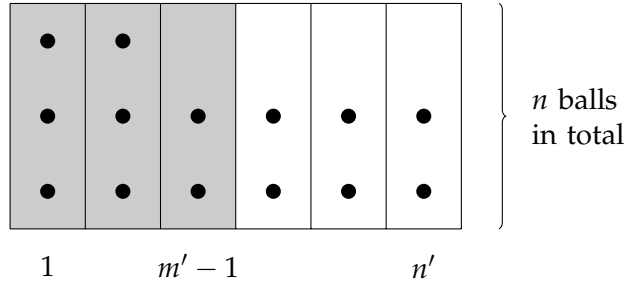
\begin{figure}[h]
  \begin{tikzpicture}
    \draw[draw=none,fill=black!20] (0,0) rectangle (3,3);
    \draw[ystep=3] (0,0) grid (6,3);
    \node at (.5,-.5) {$1$};
    \node at (2.5,-.5) {$m'-1$};
    \node at (5.5,-.5) {$n'$};
    \foreach \x in {.5,1.5,...,5.5}
      \foreach \y in {.5,1.5}
        \fill (\x,\y) circle (.1);
    \foreach \x in {.5,1.5}
      \fill (\x,2.5) circle (.1);
    \draw[decorate,decoration=brace] (6.5,3) -- (6.5,0);
    \node[anchor=west,text width=1.5cm] at (6.8,1.5) {$n$ balls in total};
  \end{tikzpicture}
  \caption{Deciding when Equation~\eqref{eq:buckets} holds.  In this
    particular example, $n=14$, $n'=6$, and $m'-1=3$.  The number of
    balls lying in the shaded region corresponds to the left-hand side
    of Equation~\eqref{eq:buckets}.\label{fig:buckets}}
\end{figure}

\begin{proof}
  We begin with the first claim.  As in the proof of
  Theorem~\ref{thm:sn}, we assume towards a contradiction that there
  is a Borel map $\psi$ carrying $(n,m)$-splitting families to
  $(n',m')$-splitting families.  We then carry out the construction of
  $\theta$, $I_k$, and $a_k$ satisfying (a) and (b) from the proof of
  Theorem~\ref{thm:sn}.  This done, we again let $\mathcal S_{n',r}$
  denote the family consisting of those $c\in2^\omega$ such that
  $\theta$ predicts $c$ at all levels which are congruent to $r$
  modulo $n'$.  We then let
  \[\mathcal S=\bigcup_{r_1<\ldots<r_{n'-m}}
  \left(\bigcap_i\mathcal S_{n',r_i}\right)\;,
  \]
  that is, the set of all $c$ such that $\theta$ predicts $c$ on at
  least $n'-m$ many congruence classes of levels.  Then it is not hard
  to verify that $\mathcal S$ is an $(n,m)$-splitting family---in
  fact, it is an $m$-splitting family.  But by the construction, no
  element of $\psi(\mathcal S)$ can split $m+1$ many of the sets
  $A_r=\set{a_{n'j+r}:j\in\omega}$ for $r<n'$.

  For the second claim, first suppose that Equation~\eqref{eq:buckets}
  holds.  We shall argue that every $(n,m)$-splitting family is in
  fact $(n',m')$-splitting, and hence the identity morphism will
  suffice.  Indeed, suppose that a family $\mathcal S$ is
  $(n,m)$-splitting and let $B_1,\ldots,B_{n'}$ be infinite subsets of
  $\omega$.  Partition each $B_i$ into either $\lfloor n/n'\rfloor$ or
  $\lfloor n/n'\rfloor+1$ many infinite subsets $C_i^j$ in such a way
  that there are $n$ many $C_i^j$ in total.  Since $\mathcal S$ is
  $(n,m)$-splitting, there exists $c\in\mathcal S$ which splits at
  least $m$ of these subsets.  It now follows from
  Equation~\eqref{eq:buckets} that $c$ splits at least $m'$ of the
  original $n'$ sets.  (To visualize this, refer to
  Figure~\ref{fig:regions}.)  Thus $\mathcal S$ is
  $(n',m')$-splitting.

  \begin{figure}[h]
    \begin{tikzpicture}
      \draw[draw=none,fill=black!20] (0,0) rectangle (3,3);
      \draw[ystep=3] (0,0) grid (6,3);
      \node at (.5,-.5) {$B_1$};
      \node at (2.5,-.5) {$B_{m'-1}$};
      \node at (5.5,-.5) {$B_{n'}$};
      \draw (0,0) grid (2,3);
      \draw[ystep=1.5] (2,0) grid (6,3);
      \draw[decorate,decoration=brace] (6.5,3) -- (6.5,0);
      \node[anchor=west,text width=1.5cm] at (6.8,1.5)
        {$n$ of the $C_i^j$ in total};
    \end{tikzpicture}
    \caption{If the shaded region contains fewer than $m$ many
      regions, then any set which splits at least $m$ regions must
      also split at least $m'$ many of the columns
      $B_i$.\label{fig:regions}}
  \end{figure}
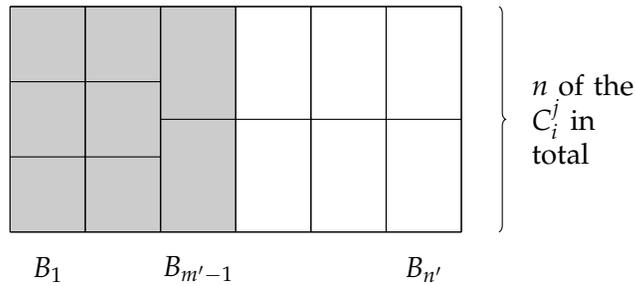

  Now suppose that Equation~\eqref{eq:buckets} fails.  Note that in
  this case we have that $n'<n$.  Once more we take the contradiction
  approach.  Suppose there is a Borel map $\psi$ which carries
  $(n,m)$-splitting families to $(n',m')$-splitting families, and
  construct $\theta$, $I_k$, and $a_k$.  We again consider the
  families $\mathcal S_{n,r}$ as above, and for $x\subset n$ we let
  \[\mathcal S_x=\bigcap_{r\notin x}\mathcal S_{n,r}
  \]
  and
  \[\mathcal S=\bigcup\set{\mathcal S_x:\abs{x\cap n'}\leq m'-1}\;.
  \]
  Then $\mathcal S$ is $(n,m)$-splitting.  Indeed, given
  $B_1,\ldots,B_n$, for each $i$ there exists $r_i<n$ such that $B_i$
  has infinite intersection with $\bigcup_j I_{nj+r_i}$.  Since
  Equation~\eqref{eq:buckets} fails, some $m$ many
  $r_{i_1},\ldots,r_{i_m}$ must lie in some $x$ with $\abs{x\cap
    n'}\leq m'-1$ (see Figure~\ref{fig:buckets}).  Then some
  $c\in\mathcal S_x$ splits $B_{i_1},\ldots,B_{i_m}$.

  On the other hand, $\psi(\mathcal S)$ is not $(n',m')$-splitting,
  since by the construction no element of $\psi(\mathcal S)$ can split
  $m'$ many of the sets $A_r=\set{a_{nj+r}:j\in\omega}$ for $r<n'$.
  This completes the proof of the second claim of
  Proposition~\ref{prop:compute}.
\end{proof}

It is not hard to see from Equation~\eqref{eq:buckets} that if $n/m$
is much larger than $n'/m'$ then there will not be a morphism from
$\fr s_{n,m}$ to $\fr s_{n',m'}$.  This will allow us to show that
there are infinite antichains among the $\fr s_{n,m}$ in the Borel
Tukey order.

\begin{thm}
  \label{thm:ac}
  The triples which define $\fr s_{2^m,m}$, where $m$ varies over the
  natural numbers $\geq3$, form an antichain in the $\leq_{BT}$
  ordering.
\end{thm}

\begin{proof}
  If $m<m'$ then there is no Borel Tukey map from $\fr s_{2^m,m}$ to
  $\fr s_{2^{m'},m'}$ by the first claim in
  Proposition~\ref{prop:compute}.  If $m>m'$, we must evaluate whether
  Equation~\eqref{eq:buckets} holds with $n=2^m$ and $n'=2^{m'}$.  It
  is not difficult to see that this equation fails, since in this case
  $r=0$ and
  \begin{align*}
    \frac{2^m}{2^{m'}}(m'-1) &= 2^{m-m'}(m'-1)\\
    &\geq((m-m')+1)(m'-1)\\
    &\geq((m-m')+1)+(m'-1)\\
    &=m\;.
  \end{align*}
  (For the second inequality we note that $AB\geq A+B$ for
  $A,B\geq2$.)  Thus it follows from the second claim in
  Proposition~\ref{prop:compute} that there is again no Borel Tukey
  map from $\fr s_{2^m,m}$ to $\fr s_{2^{m'},m'}$.
\end{proof}

Finally, we can combine members of this countable antichain to produce
more complex patterns.

\begin{cor}
  The superset ordering on $\mathcal P(\omega)$ embeds into the
  $\leq_{BT}$ ordering on triples.  In fact, this ordering embeds into
  the $\leq_{BT}$ ordering on triples that define $\fr s$.
\end{cor}

\begin{proof}
  Let us work with $\mathcal P(\omega\smallsetminus3)$ in place of
  $\mathcal P(\omega)$.  For $X\subset\omega\smallsetminus3$, we say
  that a family $\mathcal S$ is \emph{$X$-splitting} iff it is
  $(2^m,m)$-splitting for all $m\in X$.  Clearly, the ``cardinals''
  $\fr s_X$ defined by the corresponding relation are all equal to
  $\fr s$.  Moreover, if $X\supset Y$ then $X$-splitting implies $Y$
  splitting and so there is a trivial morphism from $\fr s_X$ to $\fr
  s_Y$.

  Conversely, assume that there exists $m_0\in Y\smallsetminus X$, and
  suppose towards a contradiction that there exists a Borel map $\psi$
  which carries $X$-splitting families to $Y$-splitting families.  By
  Theorem~\ref{thm:ac} and Proposition~\ref{prop:compute}, for each
  $m\in X$ there exists a $(2^m,m)$-splitting family $\mathcal
  S^{(m)}$ and a sequence of sets $A_1^{(m)},\ldots,A_{2^{m_0}}^{(m)}$
  such that no element of $\psi(\mathcal S^{(m)})$ splits $m_0$ many
  of the $A_i^{(m)}$.

  Now, if $\mathcal S=\bigcup_{m\in X}S^{(m)}$, then clearly $\mathcal
  S$ is $X$-splitting.  We claim that $\psi(\mathcal S)$ is not
  $Y$-splitting, in fact that it is not even
  $(2^{m_0},m_0)$-splitting.  To see this, note that the proof of
  Proposition~\ref{prop:compute} implies that for each $i$,
  $A_i^{(m)}$ can be taken to be of the form
  $\set{a_{nj+i}:j\in\omega}$ for some $n$.  That is, the indices are
  taken from the $i\th$ congruence class modulo $n$.  It follows
  easily that there exists a single set $A_i$ such that
  $A_i\subset^*A_i^{(m)}$ for every $m\in X$.  Now, no element of
  $\psi(\mathcal S^{(m)})$ can split $m_0$ many of the $A_i$, since
  that would imply that it splits $m_0$ many of the $A_i^{(m)}$.
  Hence $\psi(\mathcal S)$ is not $(2^{m_0},m_0)$-splitting, as
  desired.
\end{proof}

\bibliographystyle{alpha}
\begin{singlespace}
  \bibliography{tukey}
\end{singlespace}

\end{document}